\documentclass{amsart}
\usepackage{amsthm}
\usepackage{amsmath,amsfonts,amssymb,mathtools}
\usepackage{mathrsfs,mathtools,stmaryrd,wasysym}
\usepackage{enumerate}
\usepackage{xspace} 
\usepackage{esint}

\usepackage{hyperref}

\newcommand{\bu}{\mathbf{u}}
\newcommand{\bv}{\mathbf{v}}

\newcommand{\bL}{{\mathbf{L}}}
\newcommand{\bW}{{\mathbf{W}}}
\newcommand{\bM}{{\mathbf{M}}}
\newcommand{\bH}{{\mathbf{H}}}
\newcommand{\bC}{{\mathbf{C}}}

\newcommand{\bef}{{\mathbf{f}}}

\newcommand{\bz}{{\mathbf{z}}}

\newcommand{\bsigma}{{\boldsymbol{\sigma}}}

\newcommand{\bw}{{\mathbf{w}}}
\newcommand{\be}{\mathbf{e}}
\newcommand{\bX}{\mathbf{X}}
\newcommand{\bY}{\mathbf{Y}}
\newcommand{\bI}{\mathbf{I}}

\newcommand{\calG}{{\mathcal{G}}}
\newcommand{\calP}{\mathcal{P}}

\newcommand{\calA}{\mathcal{A}}
\newcommand{\calB}{\mathcal{B}}
\newcommand{\calF}{\mathcal{F}}
\newcommand{\calC}{\mathcal{C}}

\newcommand{\polL}{\mathbb{L}}
\newcommand{\polT}{\mathbb{T}}
\newcommand{\polP}{\mathbb{P}}
\newcommand{\polQ}{\mathbb{Q}}

\newcommand{\GRAD}{\nabla}
\DeclareMathOperator{\DIV}{div}
\DeclareMathOperator{\diam}{diam}

\DeclareMathOperator{\tr}{tr}
\DeclareMathOperator{\dist}{dist}

\newcommand{\diff}{\, \mbox{\rm d}}
\newcommand{\T}{\mathscr{T}}

\newcommand{\ie}{i.e.,}
\newcommand{\xR}{\mathbb{R}}
\newcommand{\xN}{\mathbb{N}}
\newcommand{\beps}{\boldsymbol{\epsilon}}
\newcommand{\point}{{\mathtt{x}}}

\theoremstyle{theorem}
\newtheorem{thrm}{Theorem}
\newtheorem{prpstn}{Proposition}
\newtheorem{crllr}{Corollary}

\theoremstyle{definition}
\newtheorem{rmrk}{Remark}
\newtheorem{dfntn}{Definition}

\DeclareMathAlphabet{\mathpzc}{OT1}{pzc}{m}{it}

\newcommand{\TheTitle}{The linear elasticity system under singular forces}
\newcommand{\ShortTitle}{Linear elasticity and singular forces}

\begin{document}
\title[\ShortTitle]{\TheTitle}

\author[A.~Allendes]{Alejandro Allendes}
\address[A.~Allendes]{Departamento de Matem\'atica, Universidad T\'ecnica Federico Santa Mar\'ia, Valpara\'iso, Chile.}
\email[A.~Allendes]{alejandro.allendes@usm.cl}

\author[G.~Campa\~na]{Gilberto Campa\~na}
\address[G.~Campa\~na]{Departamento de Ciencias, Universidad T\'ecnica Federico Santa Mar\'ia, Valpara\'iso, Chile.}
\email[G.~Campa\~na]{gilberto.campana@usm.cl}

\author[E.~Ot\'arola]{Enrique Ot\'arola}
\address[E.~Ot\'arola]{Departamento de Matem\'atica, Universidad T\'ecnica Federico Santa Mar\'ia, Valpara\'iso, Chile.}
\email[E.~Ot\'arola]{enrique.otarola@usm.cl}

\author[A.J.~Salgado]{Abner J. Salgado}
\address[A.J.~Salgado]{Department of Mathematics, University of Tennessee, Knoxville, TN 37996, USA.}
\email[A.J.~Salgado]{asalgad1@utk.edu}

\date{Draft version of \today.}



\begin{abstract}
We study the linear elasticity system subject to singular forces. We show existence and uniqueness of solutions in two frameworks: weighted Sobolev spaces, where the weight belongs to the Muckenhoupt class $A_2$; and standard Sobolev spaces where the integrability index is less than $d/(d-1)$; $d$ is the spatial dimension. We propose a standard finite element scheme and provide optimal error estimates in the $\bL^2$--norm. By proving well posedness, we clarify some issues concerning the study of generalized mixed problems in Banach spaces.
\end{abstract}

\maketitle

\section{Introduction}
\label{sec:intro}

The modeling of heterogeneous and multiscale phenomena often leads to coupled partial differential equations (PDEs) posed in different dimensions \cite{MR4219260,MR4266545,MR4040715}, or to PDEs in a single domain but where the problem data is supported on lower dimensional manifolds \cite{MR3854357,MR3734696,MR4565982,MR4250570,MR3429589,LACOUTURE2015187}. We shall not dwell in the numerous applications that fit this paradigm, but rather confine ourselves to the case of linearized elasticity with singular (in particular Dirac) sources. We refer the reader to \cite{MR4615919} for an interesting list of possible biological applications in this area.

Reference \cite{MR4615919} contains an analysis of the linear elasticity equations under Dirac sources, a numerical scheme, and some numerical illustrations. The approach to the analysis adopted by the authors of \cite{MR4615919} is via weighted Sobolev spaces and aims to extend the results of \cite{MR2888310}, dealing with a Poisson problem, to a linear elasticity system. Unfortunately, as stated in \cite{MR4060457}, the results of \cite{MR2888310} are not correct. As a result, \cite{MR4615919} also has a similar inaccuracy, so its results do not have a solid mathematical foundation.

The purpose of this paper is threefold. First, we explain the source of the inaccuracy in \cite{MR2888310,MR4615919}, in the hope that the community will avoid this pitfall in the future. Next, we provide two possible solutions, \ie, we show the well posedness of the elasticity system under singular, in particular point, forces in two functional frameworks. The first consists of weighted Sobolev spaces $\bH^1(\varpi,\Omega)$, where the weight $\varpi$ belongs to the Muckenhoupt class $A_2$; see Section~\ref{sec:Notation} for notation. The second consists of standard Sobolev spaces $\bW^{1,p}(\Omega)$, but with an integrability index $p$ smaller than $d/(d-1)$. Here, $d \in \{2, 3\}$ denotes the spatial dimension. Finally, we propose a numerical scheme consisting of piecewise linear $\bH^1(\Omega)$--conforming finite elements and, inspired by \cite{MR812624} and \cite[Section 3]{MR4408483}, provide \emph{optimal} error estimates in $\bL^2(\Omega)$.

To achieve the three goals stated above, we organize our exposition as follows. In section~\ref{sec:Notation}, we establish the notation and framework in which we will work. In section~\ref{sec:Mistake}, we discuss generalized saddle point problems in Banach spaces and show where the inaccuracy in \cite{MR2888310} and \cite{MR4615919} lies. To correct the analysis of the model, we explain two possible approaches in section~\ref{sec:Analysis}. The numerical method we propose and analyze is described in detail in section~\ref{sec:Numerics}.

\section{Notation}
\label{sec:Notation}
Let us establish the notation, the framework we will work with, and our main assumptions. Throughout our work, the spatial dimension is $d \in \{2,3\}$. The domain of interest is $\Omega \subset \xR^d$, which we assume is bounded and has at least a Lipschitz boundary $\partial\Omega$. We will use the standard notation and conventions for Lebesgue and Sobolev spaces and their norms. Many of the quantities we will deal with are tensor-valued. We will indicate this by a bold typeface. In addition, a bold script Latin letter indicates a rank-one tensor (vector), while a bold capital Latin or a bold Greek letter indicates a rank-two tensor (matrix). For $f \in L^1(E)$ with $|E| > 0$ we denote the mean value of $f$ over $E$ by $\langle f \rangle_E  = \fint_{E} f(x) \diff x:= |E|^{-1} \int_E f (x) \diff x$.

To describe the problem in which we are interested, we first provide some notation and terminology. Given a sufficiently smooth vector field $\bw : \Omega \to \xR^d$, let its symmetric gradient or its linearized strain be defined as
\[
  \beps(\bw) = \frac12 \left( \GRAD \bw + \GRAD \bw^\intercal \right).
\]
Note that by construction this is a symmetric tensor. The elastic stress tensor is then 
defined by
\[
  \bsigma(\bw) = 2\mu\beps(\bw) +\lambda \tr\beps(\bw) \bI.
\]
Here $\bI \in\xR^{d \times d}$ is the identity tensor and $\lambda$ and $\mu$ are the so-called Lam\'e constants, which we assume to be positive.

The problem that interests us in this paper is as follows:
Given a forcing term $\bef: \Omega \to \xR^d$, we must find a displacement field $\bu : \bar\Omega \to \xR^d$, which satisfies
\begin{equation}
  \label{eq:themodel}
    -\DIV \bsigma(\bu) = \bef \text{ in } \Omega, \qquad \qquad
    \bu = 0 \text{ on } \partial\Omega.
\end{equation}
Of particular interest are the cases where the forcing term is singular. This is especially the case when
$\bef$ is measure valued. Another specialization, such as the one studied in \cite{MR4615919}, is when $\bef$  is a collection of point forces: Let $K \in \xN$, $\{\point_k\}_{k=1}^K \subset \Omega$, and $\{\bef_k \}_{k=1}^K \subset \xR^d$. Then,
\begin{equation}
\label{eq:FisSumOfDeltas}
  \bef = \sum_{k=1}^K \bef_k \delta_{\point_k},
\end{equation}
where $\delta_\point$ is the Dirac distribution supported at $\point$.

Before proceeding, we recall an equivalent formulation of $\DIV \bsigma(\bw)$, for any vector field $\bw : \Omega \to \xR^d$,
which allows us to provide an analysis for problem \eqref{eq:themodel}. Namely,
\begin{equation}
\label{eq:Transform}
  \DIV \bsigma(\bw) = \mu \Delta \bw + (\mu + \lambda) \GRAD \DIV \bw.
\end{equation}

\subsection{Muckenhoupt weighted Sobolev spaces and their properties}

We could summarize the general theory of Muckenhoupt weighted Sobolev spaces, but we will refrain from unnecessary generality. A weight is a locally integrable and nonnegative function defined on $\mathbb{R}^d$. We will say that a weight $\omega$ belongs to the Muckenhoupt class $A_2$ if \cite{MR1800316,MR1774162,MR2491902}:
\begin{equation}
  [\omega]_{A_2} = \sup_{B} \fint_B \omega \diff x \fint_B \omega^{-1} \diff x < \infty,
\label{eq:weight}
\end{equation}
where the supremum is taken over all balls $B \subset \xR^d$. We call $[\omega]_{A_2}$ the Muckenhoupt characteristic of $\omega$. Note that $\omega\in A_2$ if and only if $\omega^{-1} \in A_2$ and they have the same Muckenhoupt characteristic. We define the weighted Lebesgue space $L^2(\omega,\Omega)$ as the Lebesgue space for the measure $\omega \diff x$. Since for $\omega \in A_2$, $L^2(\omega,\Omega) \subset L^1(\Omega)$, the elements of $L^2(\omega,\Omega)$ are distributions. In particular, we can speak of weighted Sobolev spaces \cite[Definition 2.1.1]{MR2491902}: 
\begin{align*}
  H^1(\omega,\Omega) &= \left\{ w \in L^2(\omega,\Omega) : \GRAD w \in \bL^2(\omega,\Omega) \right\}, \\
  \| w \|_{H^1(\omega,\Omega)} &= \left( \| w\|_{L^2(\omega,\Omega)}^2 + \| \GRAD w \|_{\bL^2(\omega,\Omega)}^2 \right)^{{\frac{1}{2}}}.
\end{align*}
This space is Hilbert. We let $H^1_0(\omega,\Omega)$ be the closure of $C_0^\infty(\Omega)$ in $H^1(\omega,\Omega)$ and note that due to the fact that $\omega \in A_2$ a Poincar\'e--type inequality holds, \ie, there exists a constant $C_P>0$ such that \cite[Theorem 1.3]{MR0643158}
\[
  \| w \|_{L^2(\omega,\Omega)} \leq C_P \| \GRAD w \|_{\bL^2(\omega,\Omega)} \quad \forall w \in H^1_0(\omega,\Omega).
\]
Consequently, $\| w \|_{H^1_0(\omega,\Omega)} = \| \GRAD w \|_{\bL^2(\omega,\Omega)}$ is an equivalent norm on $H^1_0(\omega,\Omega)$.

To conclude the discussion of weighted Sobolev spaces, we note that the following Korn--type inequalities hold \cite[Theorems 5.15 and 5.17]{MR2643399}: If $\omega \in A_2$, then there is a constant $C>0$ such that
\begin{align*}
  \| \GRAD \bw \|_{\bL^2(\omega,\Omega)} 
  & \leq C \| \beps(\bw) \|_{\bL^2(\omega,\Omega)} 
  & 
  \forall \bw &\in \bH^1_0(\omega,\Omega), 
  \\
  \| \GRAD \bw - \langle \GRAD \bw \rangle_\Omega \|_{\bL^2(\omega,\Omega)} 
  &\leq C \| \beps(\bw) - \langle \beps(\bw) \rangle_\Omega \|_{\bL^2(\omega,\Omega)} 
  &  \forall \bw &\in \bH^1(\omega,\Omega), 
  \\
  \| \GRAD \bw \|_{\bL^2(\omega,\Omega)} &\leq C  \| \beps(\bw) - \langle \beps(\bw) \rangle_\Omega \|_{\bL^2(\omega,\Omega)} \\ &+ C\diam(\Omega)^{-1} \|  \bw - \langle \bw \rangle_\Omega \|_{\bL^2(\omega,\Omega)}
  & \forall \bw &\in \bH^1(\omega,\Omega).
\end{align*}

\begin{rmrk}[other variants of Korn's inequality]
From the Korn--type inequalities presented above, it is easy to derive others. In particular, the weighted Korn's inequality that \cite{MR4615919} states as Conjecture 3.1 is valid. Since it is not needed further in our discussion, we do not elaborate on it.
\end{rmrk}

Powers of distances to lower dimensional objects are prototypical examples of Muckenhoupt weights. In particular, if we define
\begin{equation}
\label{eq:DefofMagicWeight}
  \varrho_\alpha(x) = \max_{{k \in \{ 1, \dots, K\}} } |x-\point_k|^\alpha, \qquad \alpha \in (-d,d),
\end{equation}
we have that $\varrho_\alpha \in A_2$; see \cite{MR3215609,MR1601373}.

The main reason to bring Muckenhoupt classes into our discussion is that certain Radon measures are in the dual of some weighted Sobolev spaces.

\begin{prpstn}[singular sources]
\label{prop:SingSources}
Let $\omega \in A_2$ and let $\mu \in M(\Omega)$ be a compactly supported Radon measure. If there is $r>0$ for which
\[
  \int_\Omega \int_0^r \frac{t^2 \mu( B(x,t) ) }{ \omega( B(x,t) ) } \frac{\diff t}t \diff \mu(x) < \infty,
\]
where $B(x,t)$ is the ball of center $x$ and radius $t$,
\[
  \omega(B(x,t)) = \int_{ B(x,t)} \omega(y) \diff y, \qquad 
  \mu(B(x,t)) = \int_{ B(x,t)} \diff \mu(y).
\]
Then, $\mu \in \left( H^1_0(\omega,\Omega) \right)'$. In particular, if $\alpha \in (d-2,d)$, then $\bef$, defined in \eqref{eq:FisSumOfDeltas}, belongs to $\left(\bH^1_0(\varrho_{-\alpha},\Omega) \right)'$, where $\varrho_\alpha$ is defined in \eqref{eq:DefofMagicWeight}.
\end{prpstn}
\begin{proof}
The general statement about Radon measures is from \cite[Remark 21.19]{MR2305115}. The specialization to linear combinations of Dirac deltas can be found in \cite[Proposition 5.2]{MR4081912}.
\end{proof}

Next, we note that since the collection $\{ \point_k \}_{k=1}^K \subset \Omega$ is finite, it is at a positive distance from $\partial\Omega$. Consequently, the weight $\varrho_\alpha$ defined in \eqref{eq:DefofMagicWeight} does not blow up nor degenerate near the boundary regardless of the value of $\alpha$. In fact, it is continuous. This motivates us to introduce the following restricted class of weights following \cite{MR1601373,MR3906341}.

\begin{dfntn}[class $A_2(\Omega)$]
Let $\Omega\subset\mathbb{R}^d$ be a bounded Lipschitz domain. We say that $\varpi\in A_2$ belongs to $A_2(\Omega)$ if there is an open set $\calG \subset\Omega$, and positive constants $\epsilon>0$ and $\varpi_l>0$ such that:
\begin{enumerate}[$\bullet$]
  \item A layer around the boundary $\partial\Omega$ is contained in $\calG$, \ie,
  \[
    \left\{ x\in\Omega : \dist(x,\partial \Omega)<\epsilon \right\} \subset \calG.
  \]
  
  \item The weight is strictly positive and regular in $\calG$, \ie,
  \[
    \varpi \in C(\bar\calG), \qquad \varpi(x) \geq \varpi_l \ \forall x \in \bar\calG.
  \]
\end{enumerate}
\end{dfntn}

We will follow the convention that $\omega$ denotes a weight in the class $A_2$, while $\varpi$ denotes a weight in the restricted class $A_2(\Omega)$.

\section{Generalized mixed problems in Banach spaces}
\label{sec:Mistake}

In this section, for completeness, we recall the theory of well posedness for some generalized mixed formulations \cite{MR0972452}. We do this to explain the problems in the existing literature.

Let $X,Y,M,Q$ be real reflexive Banach spaces and consider the bilinear forms
\[
  \calA : X \times Y \to \xR, \qquad
  \calB : Y \times M \to \xR, \qquad
  \calC : X \times Q \to \xR.  
\]
For any given $\calF \in Y'$ and $\calG \in Q'$, we will focus on the following problem in this section: Find $(u,p) \in X \times M$ such that
\begin{equation}
\label{eq:GenSaddle}
  \begin{dcases}
    \calA(u,v) + \calB(v,p) = \langle \calF, v \rangle & \forall v \in Y, \\
    \calC(u,q) = \langle \calG, q \rangle &\forall q \in Q,
  \end{dcases}
\end{equation}
where $\langle \cdot, \cdot \rangle$ indicates the duality pairing of the corresponding spaces. To simplify the notation we set
\begin{align*}
  K_\calB &= \ker \calB = \left\{ v \in Y : \calB(v,r) = 0  \ \forall r\in M \right\}, \\
  K_\calC &= \ker \calC = \left\{ w \in X : \calC(w,q) = 0 \ \forall q\in Q \right\}.
\end{align*}

The main result regarding problem \eqref{eq:GenSaddle} is as follows.

\begin{thrm}[well posedness]
\label{thm:CanutoMaday}
Assume that $\calA$, $\calB$, and $\calC$ are bounded, $\calF \in Y'$, and $\calG \in Q'$. Problem \eqref{eq:GenSaddle} is well posed if and only if:
\begin{enumerate}[$\bullet$]
  \item There is $\alpha_\calA > 0$ such that
  \begin{equation}
  \label{eq:infsupA1}
    \sup_{v \in K_\calB } \frac{ \calA(w,v) }{ \| v \|_Y } \geq \alpha_\calA \| w \|_X \quad \forall w \in K_\calC.
  \end{equation}
  
  \item For every $v \in K_\calB$,
  \begin{equation}
  \label{eq:infsupA2}
    \forall w \in K_\calC : \ \calA(w,v) = 0 \quad \implies \quad v=0,
  \end{equation}
  
  \item There is $\beta_\calB>0$ such that
  \begin{equation}
  \label{eq:infsupB}
    \sup_{v \in Y} \frac{\calB(v,r) }{\| v \|_Y} \geq \beta_\calB \| r \|_M \quad \forall r \in M.
  \end{equation}

  \item There is $\beta_\calC>0$ such that
  \begin{equation}
  \label{eq:infsupC}
    \sup_{w \in X} \frac{\calC(w,q) }{\| w \|_X} \geq \beta_\calC \| q \|_Q \quad \forall q \in Q.
  \end{equation}
\end{enumerate}
\end{thrm}
\begin{proof}
  See \cite[Section 2.1]{MR0972452} and \cite[Exercise 2.14]{Guermond-Ern}.
\end{proof}

\begin{rmrk}[extensions and variations]
While Theorem~\ref{thm:CanutoMaday} provides necessary and sufficient conditions for the existence of solutions to \eqref{eq:GenSaddle}, the problem itself can be generalized. See, for instance, \cite[Section 49.4]{MR4269305} for a generalization to nonreflexive spaces and \cite{MR3167999} for extensions to the case where the spaces are locally convex topological vector spaces.
\end{rmrk}

\subsection{The elasticity equation with singular sources}
Let us now discuss the problem with the proofs of \cite{MR4615919}. We show here that Lemma 3.2 in that reference unfortunately has a flaw. For this reason, the results that depend on this lemma, in particular Theorems 3.3 and 3.4 in \cite{MR4615919}, do not have a solid mathematical foundation.

Our goal in this work is not only to correct the aforementioned oversight, but we also extend the results of \cite{MR4615919}. First,  our arguments work for $d \in \{2,3\}$, while \cite{MR4615919} considers only two dimensions; see \cite[page 2356]{MR4615919}. Second, we provide a numerical scheme and an error analysis for it. The authors of \cite{MR4615919} do not derive error estimates.

Let us begin by reconciling the notation between \cite[Lemma 3.2]{MR4615919} and that used here. First, without loss of generality, \cite{MR4615919} sets $K =1$, $\point_1 = \mathtt{0}$, and considers $\varrho_\alpha = r^{2s}$ with $s \in (-1,1)$; where $r = |x|$ is the Euclidean distance from $x$ to the origin \cite[definition (3.3)]{MR4615919}. We note that in this particular case \cite[definition (3.3)]{MR4615919} can be \emph{extended} to $\varrho_\alpha = r^{ds}$, where $s \in (-1,1)$. Then, $\bH^1_s(\Omega) \longleftrightarrow \bH^1_0(\varrho_\alpha,\Omega)$ and
\[
  \polL^2_s(\Omega) = \left\{ \Sigma: \Omega \rightarrow \mathbb{R}^{d \times d}: \| \Sigma \|^2_{\polL^2_s(\Omega)} = \int_{\Omega} r^{ds} \| \Sigma \|^2 \diff x < \infty \right \} \longleftrightarrow  \bL^2(\varrho_\alpha,\Omega),
\]
where $\| \cdot \|$ denotes the Frobenius norm. Notice that, given the range of exponents, we have $\varrho_\alpha \in A_2$.

Let us now examine the proof of Lemma 3.2 in \cite{MR4615919}. The core of the problem lies in the step called \emph{Inf-sup of $a$} on page 2356, where the bilinear form $a: \polL^2_s(\Omega)\times \polL^2_{-s}(\Omega)\rightarrow \mathbb{R}$ is defined by $a(\Sigma, \tilde{\Sigma}) = \int_{\Omega} \Sigma : \tilde{\Sigma} \diff x$. Indeed, in the context of this lemma \eqref{eq:infsupA1} has to be proved where the supremum must be taken over
\[
  K_\calB = \left\{ \widetilde{\Sigma} \in \polL^2_{-s}(\Omega): b_1(\bz,\widetilde{\Sigma}) = 0 \ \forall \bz \in \bH^1_{s}(\Omega) \right\}, 
  \qquad 
  b_1(\bz,\widetilde{\Sigma}) = \int_\Omega \beps(\bz):\widetilde{\Sigma} \diff x.
\]
\ie, the supremum must be taken \emph{over the kernel} of $b_1$ and not over the whole space $\polL^2_{-s}(\Omega)$. Note that, for $\Sigma \in \polL^2_s(\Omega)$, $\widetilde{\Sigma} = r^{ds} \Sigma \in \polL^2_{-s}(\Omega) \setminus K_\calB$.
Following the same logic, the step tagged \emph{Vice versa} also has a gap.

\section{Analysis of the problem}
\label{sec:Analysis}

In this section we provide an analysis for problem \eqref{eq:themodel}. We will do this in two functional frameworks. Let $d \in \{2,3\}$, $\Omega \subset \xR^d$ be a bounded and Lipschitz domain, and let $\bX,\bY$ be spaces of vector valued functions on $\Omega$, which we assume to be Banach and reflexive. Given $\bef \in \bY'$, we seek for $\bu \in \bX$ that satisfies
\begin{equation}
\label{eq:ElastAbstract}
  \mu \int_\Omega \GRAD \bu : \GRAD \bv \diff x + (\mu + \lambda)\int_\Omega \DIV \bu \DIV \bv \diff x = \langle \bef, \bv \rangle \quad \forall \bv \in \bY.
\end{equation}
We note that we used \eqref{eq:Transform} to obtain the weak formulation.

\subsection{The elasticity equations with singular sources: Analysis in weighted spaces}
\label{sub:AnalysisWeighted}

The analysis on weighted spaces is the content of the next result.

\begin{thrm}[weighted spaces]
\label{thm:WeightedSpaces}
Let $d \in \{2,3\}$, $\Omega \subset \xR^d$ be a bounded and Lipschitz domain, and let $\varpi \in A_2$. Set $\bX = \bH^1_0(\varpi,\Omega)$ and $\bY = \bH^1_0(\varpi^{-1},\Omega)$. Assume that, either,
\begin{enumerate}[$\bullet$]
  \item $\Omega$ is a convex polyhedron, or
  \item $\varpi \in A_2(\Omega)$.
\end{enumerate}
Then, for every $\bef \in (\bH^1_0(\varpi^{-1},\Omega))'$ problem \eqref{eq:ElastAbstract} has a unique solution $\bu \in \bH^1_0(\varpi,\Omega)$, which satisfies the bound
\[
  \| \GRAD \bu \|_{\bL^2(\varpi,\Omega)} \leq C \|\bef \|_{(\bH^1_0(\varpi^{-1},\Omega))'},
  \qquad
  C>0.
\]
\end{thrm}
\begin{proof}
The proof repeats arguments that are already available in the literature. For brevity, only the main premises and steps are outlined here.

We consider each case separately, since the second case depends on the first.
\begin{enumerate}[$\bullet$]
  \item $\Omega$ \emph{is a convex polytope:} We first note that problem \eqref{eq:ElastAbstract} is well posed in $\bW^{1,p}(\Omega)$ for all $p \in (1,\infty)$; see \cite[Theorem 4.3.3]{MR2641539} on page 155 and Theorem \ref{thm:pSpaces} below. On the other hand, suitable Green's matrix estimates can be found in \cite[Theorem 5.11]{MR2641539} on page 174 and \cite[Remark 2]{MR2495783} for $d=3$. In two dimensions the above estimates, but for the Green function related to the Poisson problem, were recently derived in \cite[Lemma 2.1]{MR4060457}. As mentioned in \cite[Remark 2]{MR2495783}, these estimates can also be extended to the Green matrix of the Dirichlet problem of the Lam\'e system. Thus, based on these two ingredients, we can proceed as in \cite[Section 4]{MR4408483} to obtain the well posedness of \eqref{eq:ElastAbstract} in weighted spaces.
  
  \item $\Omega$ \emph{is a Lipschitz domain:} The case of a Lipschitz domain and $\varpi \in A_2(\Omega)$ follows the steps of the proof of \cite[Theorem 5]{MR3906341}:
\begin{enumerate}[$\circ$]
  \item An unweighted result on general domains: The proof of well posedness for $\varpi \equiv 1$ is standard and makes no requirements on the domain.
  
  \item A weighted result on smooth domains: By assuming that this subdomain is a convex polytope this is the first case.
  
  \item With the previous two results at hand the proof proceeds by localization and gluing of solutions. As mentioned earlier, the steps repeat those of \cite[Theorem 5]{MR3906341}. \qedhere
\end{enumerate}
\end{enumerate}
\end{proof}

As a consequence, we can provide the well posedness of the problem with point forces.

\begin{crllr}[point forces]
Let $d \in \{2,3\}$, $\Omega \subset \xR^d$ be a bounded Lipschitz domain, and $\alpha \in (d-2,d)$. Then, problem \eqref{eq:ElastAbstract} with $\bef$ given as in \eqref{eq:FisSumOfDeltas} has a unique solution $\mathbf{u}$ in $\bH^1_0(\varrho_\alpha,\Omega)$.
\end{crllr}
\begin{proof}
It suffices to recall that $\varrho_\alpha \in A_2(\Omega)$, use Proposition~\ref{prop:SingSources}, and Theorem~\ref{thm:WeightedSpaces}.
\end{proof}

\subsection{The elasticity equations with singular sources: Analysis in Banach spaces}
\label{sub:AnalysisBanach}

Let us now collect known results about the elasticity system in Sobolev spaces with integrability index $p \neq 2$. The analysis of this problem is as follows.

\begin{thrm}[Banach spaces]
\label{thm:pSpaces}
Let $d \in \{2,3\}$ and $\Omega \subset \xR^d$ be a bounded and Lipschitz domain. Set $\bX = \bW^{1,p}_0(\Omega)$ and $\bY = \bW^{1,q}_0(\Omega)$, where $p$ and $q$ are H\"older conjugate, \ie, $p^{-1} + q^{-1} = 1$. Assume that either
\begin{enumerate}[$\bullet$]
  \item $\Omega$ is Lipschitz and $p \in \left(\tfrac{2d}{d+1} - \varepsilon, \frac{2d}{d-1} + \varepsilon \right)$, for some $\varepsilon = \varepsilon(\Omega)>0$, or
  \item $\Omega$ is a convex polytope and $p \in (1,\infty)$.
\end{enumerate}
Then, for every $\bef \in \bW^{-1,p}(\Omega)$ problem \eqref{eq:ElastAbstract} has a unique solution $\bu \in \bW^{1,p}_0(\Omega)$, which satisfies the bound
\[
  \| \GRAD \bu \|_{\bL^p(\Omega)} \leq C \|\bef \|_{\bW^{-1,p}(\Omega)}.
\]
\end{thrm}
\begin{proof}
In the case where $\Omega$ is merely Lipschitz, the stated result is found in \cite[Section 13.6]{MR2181934} and \cite[Theorem 14]{MR1781091} for $d = 2$ and $d =3$, respectively. If $\Omega$ is a convex polyhedron, the result is found in \cite[Theorem 4.3.3]{MR2641539}.
\end{proof}

As a corollary, we obtain well posedness of the problem with point forces.

\begin{crllr}[point forces]
  Let $d \in \{2,3\}$, $\Omega \subset \xR^d$ be a bounded and Lipschitz domain, and $p \in \left(\tfrac{2d}{d+1} - \varepsilon, \tfrac{d}{d-1}\right)$ for some $\varepsilon = \varepsilon(\Omega) >0$. Then, problem \eqref{eq:ElastAbstract} with $\bef$ given as in \eqref{eq:FisSumOfDeltas} has a unique solution $\bu \in \bW^{1,p}_0(\Omega) \cap \bL^2(\Omega)$.
\end{crllr}
\begin{proof}
Note that, by assumption, $p<d/(d-1)$ so that its H\"older conjugate satisfies $q>d$. By a Sobolev embedding we have $\bW^{1,q}_0(\Omega) \hookrightarrow \bC(\bar\Omega)$. It thus follows that 
\[
 \bef \in \bM(\Omega) \hookrightarrow \bW^{-1,p}(\Omega) = \left( \bW^{1,q}_0(\Omega) \right)'
\]
and we can invoke Theorem~\ref{thm:pSpaces}. The lower bound for $p$ with, if necessary in three dimensions, a further restriction of $\varepsilon$ guarantee by a Sobolev embedding that $\bu \in \bL^2(\Omega)$. 
\end{proof}

\section{Finite element discretization}
\label{sec:Numerics}

Having presented two ways to analyze problem \eqref{eq:themodel}, we proceed to its discretization. To avoid unnecessary complications, we assume that $\Omega$ is a polytope and let $\polT = \{ \T_h\}_{h>0}$ be a family of quasi-uniform triangulations of $\bar\Omega$; see \cite{CiarletBook,Guermond-Ern,MR2373954} for precise definitions. Moreover, we assume that for every $h>0$, each $T \in \T_h$ is equivalent to either a simplex or cube.

We define the finite element spaces
\begin{equation}
\label{eq:FeSpace}
  \bX_h = \left\{ \bw_h \in \bC(\bar\Omega) : \bw_{h|T} \in \calP(T) \ \forall T \in \T_h \right\},
\end{equation}
where $\calP(T) = \polP_1(T)$ if $T$ is equivalent to a simplex or $\calP(T) = \polQ_1(T)$ if $T$ is equivalent to a cube.

The finite element scheme is formulated as follows: find $\bu_h \in \bX_h$ such that
\begin{equation}
\label{eq:Galerkin}
  \mu \int_\Omega \GRAD \bu_h : \GRAD \bv_h \diff x + (\mu + \lambda)\int_\Omega \DIV \bu_h \DIV \bv_h \diff x = \langle \bef, \bv_h \rangle \quad \forall \bv_h \in \bX_h.
\end{equation}
We note immediately that if $\bef$ is given by \eqref{eq:FisSumOfDeltas}, for every $h>0$ this problem is well posed. The existence and uniqueness of a discrete solution is not an issue. Instead, we want to show an error estimate here.

\begin{thrm}[error estimate]
Let $d \in \{2,3\}$, $\Omega$ be a bounded and convex polytope, and $\bef \in \bM(\Omega)$; in particular \eqref{eq:FisSumOfDeltas} is admissible. Assume that 
$p \in \left(\tfrac{2d}{d+2}, \tfrac{d}{d-1}\right)$. 
If $\bu \in \bW^{1,p}_0(\Omega)$ and $\bu_h \in \bX_h$ denote the solutions to \eqref{eq:ElastAbstract} and \eqref{eq:Galerkin}, respectively, then there is $C>0$ such that for all $\T_h \in \polT$ we have
\[
  \| \bu - \bu_h \|_{\bL^2(\Omega)} \leq C h^{2-\frac{d}{2}}.
\]
\end{thrm}
\begin{proof}
We proceed via a duality argument much similar to \cite{MR812624} and \cite[Section 3]{MR4408483}. We begin by noticing that $\be = \bu - \bu_h \in \bL^2(\Omega)$. Let $\bz \in \bH^1_0(\Omega)$ be the unique solution to
\begin{equation}
  \mu \int_\Omega \GRAD \bz : \GRAD \bv \diff x + (\mu + \lambda)\int_\Omega \DIV \bz \DIV \bv \diff x = \int_\Omega \be \cdot \bv \diff x \quad \forall \bv \in \bH^1_0(\Omega).
  \label{eq:z}
\end{equation}
Since $\Omega$ is a convex polytope, we have that (see \cite{MR0878829} and \cite[Theorem 3.1]{MR2019187} for $d=2$ and \cite[Section 4.3.2]{MR2641539} in page 156 for $d=3$) $\bz \in \bH^2(\Omega)$ with
\[
  |\bz|_{\bH^2(\Omega)} \leq C \| \be \|_{\bL^2(\Omega)},
\]
where the constant $C$ depends only on $\Omega$ and, possibly, the Lam\'e coefficients $\mu$ and $\lambda$. Since $\bH^2(\Omega) \hookrightarrow \bC(\bar\Omega)$, a standard interpolation estimate yields that 
\[
  \| \bz - \bz_h \|_{\bL^\infty(\Omega)} \leq C h^{2-\frac{d}{2}} |\bz|_{\bH^2(\Omega)},
\]
where the constant is independent of $\bz$ and $h$. Here, $\bz_h \in \bX_h$ corresponds to the finite element approximation of the solution $\bz$ to problem \eqref{eq:z}.

Finally, owing to the embedding $\bH^2(\Omega) \cap \bH^1_0(\Omega) \hookrightarrow \bW^{1,q}_0(\Omega)$, we may set $\bv = \bz-\bz_h$ in \eqref{eq:ElastAbstract} to arrive at
\begin{align*}
  \| \be \|_{\bL^2(\Omega)}^2 &= \mu \int_\Omega \GRAD \be : \GRAD (\bz - \bz_h) \diff x + (\mu + \lambda)\int_\Omega \DIV \be \DIV (\bz- \bz_h) \diff x \\
    &= \mu \int_\Omega \GRAD \bu : \GRAD (\bz - \bz_h) \diff x + (\mu + \lambda)\int_\Omega \DIV \bu \DIV (\bz- \bz_h) \diff x \\
    &= \langle \bef, \bz - \bz_h \rangle \leq \| \bef \|_{\bM(\Omega)} \| \bz - \bz_h \|_{\bL^\infty(\Omega)} \leq C h^{2-\frac{d}{2}} \| \be \|_{\bL^2(\Omega)},
\end{align*}
as claimed.
\end{proof}

\section*{Acknowledgement}
  AA is partially supported by ANID through FONDECYT project 1210729.
  GC is partially supported by ANID--Subdirecci\'on de Capital Humano/Doctorado Nacional/2020--21200920.
  EO is partially supported by ANID through FONDECYT project 1220156.
  AJS is partially supported by NSF grant DMS-2111228.


\end{document}